\documentclass[11pt, reqno]{amsart}

\usepackage{amscd,amsmath}
\usepackage{mathrsfs}
\usepackage{amsfonts}
\usepackage{amssymb}
\usepackage{enumerate}
\usepackage{setspace}
\usepackage{color}
\usepackage{esint}
\usepackage{dsfont}

\usepackage{tabu}
\usepackage{tikz} 
\usepackage{setspace}
\usepackage[margin=2.2cm]{geometry}

\usepackage[english]{babel}
\usepackage{graphicx}

\usepackage{xspace}

\usepackage{datetime}

\definecolor{luh-dark-blue}{rgb}{0.0, 0.313, 0.608}

\usepackage[colorlinks=true, citecolor=blue]{hyperref}
\newtheorem{theorem}{Theorem}[section]
\newtheorem{lemma}{Lemma}[section]

\newtheorem{corollary}{Corollary}[section]

\newtheorem{remark}{Remark}[section]

\newcommand{\eqn}{\begin{eqnarray}}
\newcommand{\een}{\end{eqnarray}}

\numberwithin{equation}{section}

\DeclareMathOperator{\dv}{div}

\makeatletter
 \newcommand{\norm}{\@ifstar{\@normb}{\@normi}}
 \newcommand{\@normb}[2]{\left\Vert{#1}\right\Vert_{#2}}
 \newcommand{\@normi}[2]{\Vert{#1}\Vert_{#2}}
 \makeatother

\begin{document}
\newdateformat{mydate}{\THEDAY~\monthname~\THEYEAR}

\title[Generalized Incompressible model]{On a generalized incompressible model in two dimensions}

\author[Bae]{Hantaek Bae}
\address[Hantaek Bae]{\newline Department of Mathematical Sciences, Ulsan National Institute of Science and Technology (UNIST), Republic of Korea}
\email{hantaek@unist.ac.kr}

\author[Kang]{Kyungkeun Kang}
\address[Kyungkeun Kang]{\newline Department of Mathematics, Yonsei University,, Republic of Korea}
\email{kkang@yonsei.ac.kr}

\author[Kelliher]{James P Kelliher}
\address[James P Kelliher]{\newline Department of Mathematics, University of California, Riverside, USA}
\email{kelliher@math.ucr.edu}

\author[Lee]{Woojae Lee}
\address[Woojae Lee]{\newline Department of Mathematics, Yonsei University,, Republic of Korea}
\email{woori0108@yonsei.ac.kr}

\date{\today}

\begin{abstract}
We analyze a generalized incompressible model proposed by Ohkitani \cite{Ohkitani}. This model is based on the observation that the two-dimensional Burgers' equation can be related to the incompressible Navier-Stokes equations by rotating the velocity gradient by 90 degrees. We present several results with initial data in $H^{3}$. First of all, we examine the inviscid model and show the existence and uniqueness of a local-in-time solution that blows up in finite time if the initial vorticity  contains a negative part. In the presence of viscosity, we show the existence of a unique global-in-time solution without requiring a sign condition on the initial vorticity, establish the long-time behavior of the difference between two solutions, and derive temporal decay rates for the velocity field when the initial vorticity is non-positive.
\end{abstract}

\maketitle

\vspace{-4ex}

%%%%%%%%%%%%
\section{Introduction}
%%%%%%%%%%%%
We investigate a generalized incompressible model in $\mathbb{R}^2$ proposed in \cite{Ohkitani}  which provides a framework of  the transition from Burgers' equation to the incompressible Navier-Stokes equation, aiming to offer insights into the complex behavior of fluid flows. This generalization is achieved by rotating the gradient of the velocity of Burgers' equation, which allows for both theoretical and numerical analysis  of the interpolating equations. The resulting system is expected to be well-posed for a certain range of angle parameters; however, establishing this rigorously is nontrivial because the model lacks both the maximum principle available for Burgers' equation and the favorable vorticity structure inherent to the incompressible Navier–Stokes equations. Before presenting the model in \cite{Ohkitani}, we first recall some basic properties of these two equations. 

\vspace{1ex}

The incompressible Navier-Stokes equations are given by 
\eqn \label{NSE}
u_{t}  +u\cdot \nabla u +\nabla p-\nu\Delta u=0,\quad \dv u=0,
\een
where $u$ is the fluid velocity, $p$ is the pressure, and $\nu$ is the constant of viscosity. On the other hand, Burgers' equation is  given by 
\eqn \label{Burgers}
\overline{u}_{t}  +\overline{u}\cdot \nabla \overline{u} -\nu\Delta \overline{u}=0,
\een
where we use the same viscosity constant to link (\ref{Burgers}) to (\ref{NSE}) below.  When $\nu\geq0$, (\ref{NSE}) and (\ref{Burgers}) share the same local well-posedness result and the blow-up criterion in three dimensions.  More precisely, let $\mathcal{V}=u$ or $\mathcal{V}=\overline{u}$ with $\mathcal{V}_{0}\in H^3$. Then, there exists $T>0$ such that each system admits a unique solution  satisfying 
\[
\left\|\mathcal{V}(t)\right\|^{2}_{H^3}+\nu\int^{t}_{0}\left\|\nabla \mathcal{V}(s)\right\|^{2}_{H^3}ds\leq C \left\|\mathcal{V}_{0}\right\|^{2}_{H^3}
\]
on the time interval $[0,T)$. Moreover, the maximal existence time $T^{\ast}<\infty$ if and only if 
\eqn \label{Blowup criterion}
\limsup_{t\nearrow T^{\ast}}\int^{t}_{0}\left\|\mathcal{V}(t)\right\|^2_{L^{\infty}}dt=\infty.
\een
(Following \cite{BV, Majda}, we refer to $v$ as a strong solution and assume it throughout the paper.)  Since (\ref{Burgers}) satisfies the maximum principle $\|\overline{u}(t)\|_{L^{\infty}} \leq \|\overline{u}_{0}\|_{L^{\infty}}$, (\ref{Burgers}) is globally well-posed with initial data in $H^3$. In fact, we can reduce the required regularity of the initial data from $H^3$ to $H^{\frac{1}{2}}$ for (\ref{Burgers}) \cite{Pooley Robinson}. However, we do not know whether (\ref{Blowup criterion}) is excluded for (\ref{NSE}), mainly due to the presence of the pressure. (There are several modified models of (\ref{NSE}) in both directions showing the global well-posedness \cite{Chae, Pooley} or the finite time singularity formation \cite{Cheskidov, Friedlander Pavlovic, Gallagher Paicu, Montgomery Smith, Tao}.)

\vspace{1ex}

In two dimensions, both (\ref{NSE}) and (\ref{Burgers}) are globally well-posed with initial data in $H^3$ when $\nu > 0$. However, when $\nu = 0$, (\ref{NSE}) and (\ref{Burgers}) exhibit different behaviors. We note that (\ref{NSE}) remains globally well-posed with initial data in $H^3$ even when $\nu = 0$ which corresponds to the incompressible Euler equations. Let $\omega = \partial_1 u_2 - \partial_2 u_1$ be the vorticity of $u$. Since $\omega$ satisfies the following transport equation
\[
\omega_{t}+u\cdot \nabla \omega=0,
\]
$\omega$ satisfies the maximum principle $\|\omega(t)\|_{L^{\infty}}= \|\omega_{0}\|_{L^{\infty}}$. So, the global well-posedness of (\ref{NSE}) with $\nu=0$ can be obtained using the Beale-Kato-Majda criterion \cite{BKM}. In contrast, the vorticity $\overline{\omega}=\partial_1 \overline{u}_2 - \partial_2 \overline{u}_1$ of  (\ref{Burgers}) satisfies 
\[
\overline{\omega}_{t}+\overline{u}\cdot \nabla \overline{\omega}=-(\dv \overline{u})\overline{\omega}
\]
and thus $\overline{\omega}$ fails to satisfy the maximum principle. In fact, finite-time singularity formation occurs for (\ref{Burgers}) when $\nu=0$ (see the Appendix of \cite{Ohkitani}).

\vspace{1ex}

While (\ref{NSE}) and (\ref{Burgers}) have been compared in the literature, both theoretically in \cite{kiselev1957existence, Pooley Robinson} and numerically in \cite{ohkitani2012numerical}, a continuous interpolation between the two systems via an explicit transformation does not appear to have been previously established. In this direction, Ohkitani introduces a generalized incompressible model  \cite{Ohkitani} as follows. Firstly, to reformulate (\ref{Burgers}) as an incompressible equation like (\ref{NSE}), we consider (\ref{Burgers}) under the assumption of potential flow, where the velocity field is defined as $\overline{u} = \nabla \psi$. Since 
\[
\overline{u}\cdot \nabla \overline{u}=\frac{1}{2}\nabla|\overline{u}|^2+\overline{\omega} \begin{pmatrix}
-\overline{u}_2 \\
\overline{u}_1
\end{pmatrix}, \quad \overline{\omega} = \partial_1 \overline{u}_2 - \partial_2 \overline{u}_1=\partial_{1}\partial_{2}\psi-\partial_{2}\partial_{1}\psi=0,
\]
$\psi$ (at least formally) satisfies 
\[
\psi_{t}+\frac{|\nabla \psi|^2}{2}-\nu\Delta \psi=C(t)
\]
for some function $C(t)$. Let $\widetilde{u}=\nabla^{\perp}\psi$, where $\nabla^{\perp}=(-\partial_{2}, \partial_{1})$. Since 
\[
\nabla^{\perp}\frac{|\nabla \psi|^2}{2}=\nabla^{\perp}\frac{|\nabla^{\perp} \psi|^2}{2}=\nabla^{\perp}\psi \cdot \nabla^{\perp} (\nabla^{\perp}\psi),
\]
$\widetilde{u}$ satisfies  
\eqn \label{Incompressible Burgers eq}
\widetilde{u}_{t}  +\widetilde{u}\cdot \nabla^{\perp} \widetilde{u}-\nu\Delta \widetilde{u}=0, \quad \dv \widetilde{u}=0.
\een

\vspace{1ex}

By interpolating (\ref{NSE}) and (\ref{Incompressible Burgers eq}), the following model is proposed in \cite{Ohkitani}: 
\eqn \label{Eq of v}
v_{t}+\mathbb{P}\left(v\cdot \mathcal{D}^{\alpha}v\right)-\nu\Delta v=0,\quad \dv v=0, \quad 0 < \alpha < \frac{\pi}{2}
\een
where $\mathbb{P}$ is the Leray projection operator onto the space of divergence-free vector fields  and 
\[
\mathcal{D}^{\alpha}=\mathcal{R}_{\alpha}\nabla =\begin{pmatrix}
\cos\alpha &-\sin\alpha\\
\sin\alpha&\cos\alpha
\end{pmatrix} \begin{pmatrix}
\partial_{1}\\
\partial_{2}
\end{pmatrix}=\begin{pmatrix}
[\cos\alpha]\partial_{1}-[\sin\alpha]\partial_{2}\\
[\sin\alpha]\partial_{1}+[\cos\alpha]\partial_{2}
\end{pmatrix}.
\]
It is clear that $\alpha=0$ corresponds to (\ref{NSE}) and $\alpha=\frac{\pi}{2}$  to (\ref{Incompressible Burgers eq}). The well-posedness results discussed above and the numerical simulations performed in \cite{Ohkitani} lead to two questions: first, whether local-in-time solutions develop finite-time singularities when $\nu = 0$; and second, whether \eqref{Eq of v} is globally well-posed for $\nu > 0$.  In this paper, we address both questions analytically with initial data in $H^{3}$ where we choose this space to handle the vorticity $\omega$ in $L^{\infty}$. 

\vspace{1ex}

This paper is organized as follows. After providing preliminary material in Section \ref{sec:2}, we investigate the inviscid case ($\nu=0$) in Section \ref{sec:3} to show two results: the existence of a unique local-in-time solution (Theorem \ref{LWP theorem}) and the finite-time blow-up of the solution under the condition that the initial vorticity is non-negative at least at one point (Theorem \ref{Blow-up Theorem}). Sections \ref{sec:4} and \ref{sec:5} are devoted to the viscous case ($\nu>0$). In Section \ref{sec:4}, we prove the existence of a unique global-in-time  solution without assuming a sign condition on the vorticity (Theorem \ref{GWP theorem}), and we establish the long-time behavior of the difference $v_{1}-v_{2}$, where $v_{1}$ and $v_{2}$ are the solutions corresponding to $\alpha_{1}$ and $\alpha_{2}$, respectively (Theorem \ref{Theorem Difference bound}). Finally, in Section \ref{sec:5}, we derive temporal decay rates of the global solutions constructed in Section \ref{sec:4} under the assumption that the initial vorticity is non-positive and the initial velocity belongs to $H^{3}\cap L^1$ (Theorem \ref{L2 Decay theorem}) or to $H^{3}\cap \dot{H}^{-\sigma}$ (Theorem \ref{L2 Decay theorem 2}).

%%%%%%%%%%%%%%
\section{Preliminaries}\label{sec:2}
%%%%%%%%%%%%%%
All constants will be denoted by $C$, and we follow the convention that such constants may vary from line to line, or even between two occurrences within the same expression. We also use a simplified notation for integrals over the spatial variables:
\[
\int f=:\int_{\mathbb{R}^{2}}f(x)dx.
\]

%The space $L^{p}(\mathbb{R}^2)$, $1\leq p\leq \infty$, consists of measurable functions equipped with the following norm
%\[
%\left\|f\right\|_{L^{p}}=\left(\int_{\mathbb{R}^2} |f(x)|^{p}\,dx\right)^{\frac{1}{p}} \quad \text{when $1\leq p<\infty$},\quad \left\|f\right\|_{L^{\infty}}=\esssup_{x\in\mathbb{R}^2} |f(x)|. 
%\]

Let $f$ be a scalar function in $\mathbb{R}^{2}$. (When $F$ is a vector in $\mathbb{R}^{2}$, we apply the following definitions to  each component of $F$.) Give a multi-index $\beta=(\beta_{1}, \beta_{2})$, we define 
\[
\nabla^{\beta}f=\frac{\partial^{|\beta|}f}{\partial^{\beta_{1}}_{1} \partial^{\beta_{2}}_{2}}, \quad |\beta|=\beta_{1}+ \beta_{2}.
\]
When $k$ is a nonnegative integer, we set $D^{k}f=\left\{\nabla^{\beta}f: |\beta|=k\right\}$.  For $k\in \mathbb{N}$, $H^{k}$ and $\dot{H}^k$ denote the inhomogeneous and homogeneous Sobolev spaces equipped with the following norm and semi-norm, respectively 
\[
\|f\|^{2}_{H^{k}}=\|f\|^{2}_{L^{2}}+\sum^{k}_{i=1}\left\|D^{i} f\right\|^{2}_{L^{2}}, \quad \quad \|f\|_{\dot{H}^k}=\left\|D^{k} f\right\|_{L^{2}}.
\]
 $\Lambda^s$ stands for the Fourier multiplier operator given by $|\xi|^s$. This definition allows the Sobolev index to be extended to real (in particular, negative) values $s\in\mathbb{R}$ by
\[
 \|f\|_{\dot{H}^s}=\left\|\Lambda^{s} f\right\|_{L^{2}}.
\]
We now recall some inequalities.

\begin{enumerate}[]
\item \textbullet \ Gronwall's inequality \cite[Page 624]{Evans}: Let $\eta$ be a nonnegative, absolutely continuous function on $[0,T]$ satisfying, for a.e. $t$, the differential inequality $
\eta'(t) \leq \phi(t)\eta(t) +\psi(t)$, where $\phi(t)$ and $\psi(t)$ are nonnegative  and integrable functions on $[0,T]$. Then, 
\[%\label{Gronwall}
\eta(t) \leq \left(\eta(0)+\int^{t}_{0}\psi(\tau)d\tau\right)\exp\left[\int^{t}_{0}\phi(\tau)d\tau\right], \quad 0\leq t\leq T.
\]%\een
\item \textbullet \ Ladyzhenskaya  and Sobolev inequalities \cite{Ladyzhenskaya, Nirenberg}: 
\begin{subequations} \label{Lp bounds}
\begin{align}
&\left\|f\right\|_{L^{4}}\leq C \left\|f\right\|^{\frac{1}{2}}_{L^{2}}\left\|f\right\|^{\frac{1}{2}}_{\dot{H}^{1}},\label{L4 bound}\\
&\|f\|_{L^{\infty}}\leq C \|f\|_{H^{2}}. \label{L infty bound}
\end{align}
\end{subequations}
\item \textbullet \ Interpolations \cite{Nirenberg}: for $0<s_{0}<s<s_{1}$,
\eqn \label{Interpolation}
\left\|D^{s}f\right\|_{L^{2}}\leq C \left\|D^{s_{0}}f\right\|^{\theta}_{L^{2}} \left\|D^{s_{1}}f\right\|^{1-\theta}_{L^{2}}, \quad \theta=\frac{s_{1}-s}{s_{1}-s_{0}}.
\een
\item \textbullet \ Product rules: (i)  \cite{Folland}: for $k>1$
\eqn
\|fg\|_{H^{k}}\leq C \|f\|_{H^{k}}\|g\|_{H^{k}}, \label{Product estimates}
\een
and (ii) \cite{Kato} for $1< p<\infty$ and $p_{i}, q_{i} \ne 1$, $i=1,2$,
\eqn \label{Product estimates 2}
\left\|D^{s}(fg)\right\|_{L^{p}} \le C \left(\left\|D^{s}f\right\|_{L^{p_{1}}} \|g\|_{L^{q_{1}}} +\|f\|_{L^{p_{2}}} \left\|D^{s}g\right\|_{L^{q_{2}}} \right), \quad \frac{1}{p}=\frac{1}{p_{1}}+\frac{1}{q_{1}}=\frac{1}{p_{2}}+\frac{1}{q_{2}}.
\een 
\end{enumerate}

We also recall the following commutator estimates.

\begin{lemma} \cite[Corollary 5.3]{Li} \label{Commutator lemma}\upshape
Let $[D^s,f]g=D^s(fg)-fD^s g$ with $s\geq 1$. For $1<p<\infty$ and $1<p_{1},p_{2},p_{3},p_{4}\leq\infty$,
\[
\left\|[D^s,f]g\right\|_{L^{p}}\leq C \left\|D^{s} f\right\|_{L^{p_{1}}}\left\|g\right\|_{L^{p_{2}}}+C\left\|\nabla f\right\|_{L^{p_{3}}}\left\|D^{s-1}g\right\|_{L^{p_{4}}}, \quad \frac{1}{p}= \frac{1}{p_{1}}+\frac{1}{p_{2}}=\frac{1}{p_{3}}+\frac{1}{p_{4}}.
\]
\end{lemma}

Since the role of vorticity $\omega=\partial_1 v_2 - \partial_2 v_1$ is critical in this paper, we derive the governing equation for the vorticity before presenting our main results. We first note that
\eqn \label{New advection term}
v\cdot \mathcal{D}^{\alpha}v=\mathcal{R}_{-\alpha}(v)\cdot \nabla v,
\een
where
\eqn \label{Rotation operator}
\mathcal{R}_{-\alpha}(v)=\begin{pmatrix}
\cos\alpha &\sin\alpha\\
-\sin\alpha&\cos\alpha
\end{pmatrix} \begin{pmatrix}
v_{1}\\
v_{2}
\end{pmatrix}=\begin{pmatrix}
[\cos\alpha]v_{1}+[\sin\alpha]v_{2}\\
[-\sin\alpha]v_{1}+[\cos\alpha] v_{2}
\end{pmatrix}.
\een
By  taking the curl operator to \eqref{New advection term}, we have  
\[
\nabla \times \left(v\cdot \mathcal{D}^{\alpha}v\right)=\nabla \times \left(\mathcal{R}_{-\alpha}(v)\cdot \nabla v\right)=\mathcal{R}_{-\alpha}(v)\cdot \nabla \omega+\partial_{1}\mathcal{R}_{-\alpha}(v)\cdot \nabla v_{2}-\partial_{2}\mathcal{R}_{-\alpha}(v)\cdot \nabla v_{1}
\]
with 
\[
\begin{split}
&\partial_{1}\mathcal{R}_{-\alpha}(v)\cdot \nabla v_{2}-\partial_{2}\mathcal{R}_{-\alpha}(v)\cdot \nabla v_{1}\\
&=\partial_{1}\mathcal{R}_{-\alpha}(v)_{1} \partial_{1}v_{2}+\partial_{1}\mathcal{R}_{-\alpha}(v)_{2} \partial_{2}v_{2}-\partial_{2}\mathcal{R}_{-\alpha}(v)_{1} \partial_{1}v_{1}-\partial_{2}\mathcal{R}_{-\alpha}(v)_{2} \partial_{2}v_{1}\\
&=[\cos\alpha] \left[\partial_{1}v_{1}\partial_{1}v_{2}+\partial_{1}v_{2}\partial_{2}v_{2}-\partial_{1}v_{1}\partial_{2}v_{1}-\partial_{2}v_{1}\partial_{2}v_{2}\right]+[\sin\alpha] \left[\left(\partial_{1}v_{2}\right)^2+\left(\partial_{2}v_{1}\right)^2-2\partial_{1}v_{1}\partial_{2}v_{2}\right]\\
&=[\cos\alpha] \left[\partial_{1}v_{1}\left(\partial_{1}v_{2}-\partial_{2}v_{1}\right)+\partial_{2}v_{2}\left(\partial_{1}v_{2}-\partial_{2}v_{1}\right)\right]+[\sin\alpha] \left[\left(\partial_{1}v_{2}\right)^2+\left(\partial_{2}v_{1}\right)^2-2\partial_{1}v_{1}\partial_{2}v_{2}\right]\\
&=[\cos\alpha](\dv v)\omega+[\sin \alpha] \left[\left(\partial_{1}v_{2}\right)^2+\left(\partial_{2}v_{1}\right)^2+\left(\partial_{1}v_{1}\right)^2+\left(\partial_{2}v_{2}\right)^2\right]=(\sin\alpha)|\nabla v|^2,
\end{split}
\]
where we use $\dv v=0$ to have the last equality. Thus, $\omega$ satisfies
\eqn \label{Eq of omega}
\omega_{t}+\mathcal{R}_{-\alpha}(v)\cdot \nabla \omega-\nu\Delta \omega=-[\sin\alpha]|\nabla v|^2.
\een
This is equivalent to \cite[Eq (10)]{Ohkitani}
\[
\omega_{t}+[\sin\alpha]\Delta \frac{|v|^2}{2}+[\cos\alpha]v\cdot \nabla \omega-\nu\Delta \omega=0,
\]
but we find that (\ref{Eq of omega}) is more useful than this for proving our results. We finally notice  that 
\eqn\label{div curl}
\dv \mathcal{R}_{-\alpha}(v)=[\sin\alpha]\omega.
\een

%%%%%%%%%%%%%%%%%%%%%%%%%
\section{Inviscid case} \label{sec:3}
%%%%%%%%%%%%%%%%%%%%%%%%%
We recall (\ref{Eq of v}) with $\nu=0$: 
\eqn \label{Inv Eq of v}
v_{t}+\mathbb{P}\left(\mathcal{R}_{-\alpha}(v)\cdot \nabla v\right)=0,
\een
where we use (\ref{New advection term}) to express the nonlinear term.

\begin{theorem} \label{LWP theorem}\upshape
Let $v_{0}\in H^{3}$ with $\dv v_{0}=0$. Then,  there exists a unique solution $v\in C([0,T);H^3)$ of (\ref{Inv Eq of v}) satisfying the following bound:
\eqn \label{Inviscid Bound}
\left\|v(t)\right\|_{H^{3}}\leq \frac{\left\|v_{0}\right\|_{H^{3}}}{1-Ct\left\|v_{0}\right\|_{H^{3}}}, \quad 0\leq t<T=\frac{1}{C\left\|v_{0}\right\|_{H^{3}}}.
\een
Moreover, the maximal existence time $T^{\ast}<\infty$ if and only if 
\[
\limsup_{t\nearrow T^{\ast}}\int^{t}_{0}\left\|\omega(\tau)\right\|_{L^{\infty}}d\tau=\infty.
\]
\end{theorem}

\begin{proof}
To prove Theorem \ref{LWP theorem}, we focus on deriving a uniform bound for $\|v\|_{H^{3}}$: the existence of a unique solution can be established by applying the methods in \cite{BV, Majda}. By applying the derivative $D^k$, $k=0,1,2,3$, to (\ref{Inv Eq of v}) and taking $L^2$ inner product of the resulting equation with $D^{k}v$, we have 
\eqn \label{energy estimates 1}
\begin{split}
\frac{1}{2}\frac{d}{dt}\left\|v\right\|^{2}_{H^{3}}&=-\sum^{3}_{k=0}\int D^{k}(\mathcal{R}_{-\alpha}(v) \cdot \nabla v)\cdot D^{k}v\\
&=-\sum^{3}_{k=0}\int \left(\left[D^{k}, \mathcal{R}_{-\alpha}(v)\right]\cdot \nabla v\right)\cdot D^{k}v -\sum^{3}_{k=0}\int \left(\mathcal{R}_{-\alpha}(v)\cdot \nabla D^{k}v\right) \cdot D^{k}v\\
&=-\sum^{3}_{k=0}\int \left(\left[D^{k}, \mathcal{R}_{-\alpha}(v)\right]\cdot \nabla v\right)\cdot D^{k}v  +\frac{1}{2}\sum^{3}_{k=0}\int \dv \mathcal{R}_{-\alpha}(v) \left|D^{k}v\right|^{2}\\
&=-\sum^{3}_{k=0}\int \left(\left[D^{k}, \mathcal{R}_{-\alpha}(v)\right]\cdot \nabla v\right)\cdot D^{k}v  +\frac{[\sin\alpha]}{2}\sum^{3}_{k=0}\int \omega \left|D^{k}v\right|^{2},
\end{split}
\een
where we use (\ref{div curl}) to the last term on the right-hand side of (\ref{energy estimates 1}). By applying Lemma \ref{Commutator lemma}, we obtain
\eqn \label{Bound of v time derivative 1}
\frac{d}{dt}\left\|v\right\|^{2}_{H^{3}}\leq C\left\|\nabla v\right\|_{L^{\infty}}\left\|v\right\|^{2}_{H^{3}},
\een
where we use that $|D^k\mathcal{R}_{-\alpha}(v)|\leq C |D^{k}v|$ holds pointwise. Since $\left\|\nabla v\right\|_{L^{\infty}}\leq C \left\|v\right\|_{H^{3}}$ by (\ref{L infty bound}), we can deduce the following inequality: 
\[
\frac{d}{dt}\left\|v\right\|_{H^{3}}\leq C\left\|v\right\|^{2}_{H^{3}}.
\]
By solving this ODE inequality, we obtain (\ref{Inviscid Bound}). Moreover, by applying the classical Beale-Kato-Majda criterion \cite{BKM} to (\ref{Bound of v time derivative 1}), we have  
\[
\frac{d}{dt}\left\|v\right\|^{2}_{H^{3}}\leq C\left[1+ \left(1+\log^+ \left\|v\right\|_{H^{3}}\right) \left\|\omega\right\|_{L^{\infty}}+\left\|\omega\right\|_{L^{2}}\right] \left\|v\right\|^{2}_{H^{3}}.
\]
Moreover, 
\[
\frac{1}{2}\frac{d}{dt}\left\|\omega\right\|^{2}_{L^{2}}=[\sin\alpha]\int\omega^3 -\int [\sin\alpha]|\nabla v|^2\omega \leq C \left\|\omega\right\|_{L^{\infty}}\left\|\omega\right\|^{2}_{L^{2}},
\]
where we use $\left\|\nabla v\right\|_{L^{2}}\leq C \left\|\omega\right\|_{L^{2}}$ because $\dv v=0$ \cite{Stein}. From these two inequalities, we conclude that the maximal existence time $T^{\ast}<\infty$ if and only if 
\[
\limsup_{t\nearrow T^{\ast}}\int^{t}_{0}\left\|\omega(\tau)\right\|_{L^{\infty}}d\tau=\infty.
\]
This completes the proof of Theorem \ref{LWP theorem}. 
\end{proof}

To show that the solution from Theorem \ref{LWP theorem} fails to exist globally-in-time, we impose a sign condition on the initial vorticity.

\begin{theorem} \label{Blow-up Theorem}\upshape
Let $v_{0}\in H^{3}$ with $\dv v_{0}=0$. Then, the solution constructed in Theorem \ref{LWP theorem} cannot exist globally in time if $\displaystyle \min_{x\in\mathbb{R}^2}\omega_{0}(x)<0$. Moreover, the blow-up time $T^*$ is bounded above as follows:
\eqn \label{upper bound of T}
 T^*\leq -\frac{2}{[\sin\alpha]\displaystyle \min_{x\in\mathbb{R}^2}\omega_{0}(x)}.
\een
\end{theorem}

\begin{proof} 
W first find the following inequality from (\ref{Eq of omega}):
\eqn \label{omega inequality}
\omega_{t}+\mathcal{R}_{-\alpha}(v)\cdot \nabla \omega \leq -\frac{[\sin\alpha]}{2}\omega^2,
\een
because $|\omega|^2 \leq 2|Dv|^2$ holds pointwise. Let $\eta$ be the flow map associated with $\mathcal{R}_{-\alpha}(v)$:
\[
\frac{d}{dt}\eta(t, y)=\mathcal{R}_{-\alpha}(v)(t, \eta(t,y)), \quad \eta(t,y)\Big|_{t=0}=y. 
\]
Then, along the map $\eta$, $W(t,y)=\omega(t, \eta(t,y))$ satisfies 
\[
\frac{d}{dt}W(t,y) \leq -\frac{[\sin\alpha]}{2}W^2(t,y).
\]
By solving this inequality, we derive
\[
W(t,y)\leq  \frac{2\omega_{0}(y)}{\displaystyle 2+[\sin\alpha] \omega_{0}(y)t}.
\]
So, if initially $\omega_{0}(y)<0$, then $W(t,y)$ tends to $-\infty$ as
\[
t\rightarrow  -\frac{2}{[\sin\alpha]\omega_{0}(y)}.
\]
By the blow-up criterion in Theorem \ref{LWP theorem}, the solution cannot exist globally-in-time, with the maximal existence time $T^{\ast}$ bounded  as (\ref{upper bound of T}).
\end{proof}

%%%%%%%%%%%%%%%%%%%%%%%%%%%
\section{The viscous case} \label{sec:4}
%%%%%%%%%%%%%%%%%%%%%%%%%%%

%%%%%%%%%%%%%%%%%%%%
\subsection{Global well-posedness}
%%%%%%%%%%%%%%%%%%%
We recall (\ref{Eq of v}) with $\nu>0$:
\eqn \label{Vis Eq of v}
v_{t}+\mathbb{P}\left(\mathcal{R}_{-\alpha}(v)\cdot \nabla v\right)-\nu\Delta v=0.
\een
Since the value of $\nu$ does not affect our results, we set $\nu=1$ for simplicity. To show the global existence of solutions without any sign condition on the initial vorticity, we introduce a lemma concerning the positive part $\omega^+ = \max(0, \omega)$. To estimate $\omega^{+}$, we use the following two lemmas.

\begin{lemma}\upshape \label{L2 lemma plus}
For a scalar function $f \in L^{2}(0,\infty; L^2)$ with $f_{t} \in L^{2}(0,\infty; L^2)$, we have
\[
\frac{1}{2}\frac{d}{dt}\left\|f^{+}\right\|^{2}_{L^{2}}=\int f_{t}f^{+},
\]
where the time derivative is taken in the distributional sense.
\end{lemma}

\begin{proof}
The proof follows from a classical mollification together with another standard approximation of the positive part of $f$. Let 
\[
G_\epsilon(s) = \frac{\sqrt{s^2 + \epsilon^2} + s - \epsilon}{2} \quad \text{and} \quad G_\epsilon'(s) = \frac{1}{2} \left( \frac{s}{\sqrt{s^2 + \epsilon^2}} + 1 \right) \quad \text{for $\epsilon>0$}.
\]
Then, $G_{\epsilon}\in C^{\infty}$ and $G_{\epsilon} \rightarrow \max\{0, s\}$ as $\epsilon\rightarrow 0$. Moreover, $0<G_\epsilon'(s)<1$ for all $s\in\mathbb{R}$ and
\[
G_\epsilon'(s)\to
\begin{cases}
1, & s>0,\\
\frac{1}{2} & s=0,\\
0 & s<0,
\end{cases}
\qquad \text{as } \epsilon\to 0.
\]For this approximation $G_\epsilon$, the function 
\[
t \mapsto \frac{1}{2} \int |G_\epsilon(f(x,t))|^2
\]
is absolutely continuous and thus we can take the time derivative for a.e. $t\in (0,\infty)$ to derive 
\[
\frac{d}{dt} \left( \frac{1}{2} \int |G_\epsilon(f)|^2 \right) =\int f_t \left( G_\epsilon(f) G_\epsilon'(f) \right).
\]
Since 
\[
G_\epsilon(f) \to f^+, \quad G_\epsilon'(f) \to \chi_{\{f>0\}}+\frac{1}{2}\chi_{\{f=0\}}
\]
as $\epsilon \to 0$,  
\[
G_\epsilon(f) G_\epsilon'(f) \to f^+ \chi_{\{f>0\}} = f^+
\]
pointwise. Since $|G_\epsilon(f) G_\epsilon'(f)| \leq |f^+|$ with $f^+ \in L^2$ and $f_{t} \in L^2$, we obtain
\[
\lim_{\epsilon\rightarrow 0}\int f_t G_\epsilon(f) G_\epsilon'(f)= \int f_t f^+ 
\]
by the Lebesgue dominated convergence theorem. We also note 
\[
\lim_{\epsilon\rightarrow 0} \int |G_\epsilon(f)|^2 =\int |f^+|^2
\]
as $\epsilon\rightarrow 0$.
Integrating the identity  over $[0, \tau]$ gives 
\[
\frac{1}{2} \|G_\epsilon(f(\tau))\|_{L^2}^2 - \frac{1}{2} \|G_\epsilon(f(0))\|_{L^2}^2 = \int_0^\tau \int f_t \left( G_\epsilon(f) G_\epsilon'(f) \right)dt.
\]
Passing to the limit $\epsilon \to 0$ on both sides yields 
\[
\frac{1}{2} \|f^+(\tau)\|_{L^2}^2 - \frac{1}{2} \|f^+(0)\|_{L^2}^2 = \int_0^\tau \int f_t f^+ dt.
\]
Differentiating with respect to $\tau$ (in the sense of distributions) completes the proof of Lemma \ref{L2 lemma plus}.
\end{proof}

\begin{lemma}\upshape \label{L2 lemma grad plus}
For a scalar function $f\in H^1$, the following identity holds:
\[
-\int f^+ \Delta f  = \left\| \nabla f^+ \right\|_{L^2}^2.
\]
\end{lemma}

\begin{proof}
If $f$ is in $H^{1}$, $f^{+}$ is also in $H^{1}$ by Stampacchia's Lemma \cite[Theorem A.1]{Kinderlehrer}. Moreover, the weak gradient of $f^+$ is given almost everywhere by $\nabla f^+ = \chi_{\{f > 0\}} \nabla f $, which justifies the following identity
\[
\int \nabla f^+ \cdot \nabla f= \int_{\{f>0\}} \nabla f \cdot \nabla f = \int \left|\nabla f^+\right|^2.
\]
This completes the proof of Lemma {\ref{L2 lemma grad plus}}. 
\end{proof}

We now show that $\omega^+$ decays in time as follows.

\begin{lemma}\label{Decay lemma of omega} \upshape
Suppose there exists a global-in-time solution $v\in C([0,\infty); H^{3})$ of (\ref{Vis Eq of v}) with $v_0\in H^3$. Then, $\omega^+$ decays in time as follows: 
\eqn \label{Decay of positive vorticity}
\|\omega^+(t)\|_{L^\infty}\leq \frac{2\|\omega_0\|_{L^\infty}}{(2+[\sin\alpha]\|\omega_0\|_{L^\infty}t)}.
\een
\end{lemma}

\begin{proof}
As (\ref{omega inequality}), we use  
\eqn\label{eq;lemvor}
\omega_t-\Delta\omega+\mathcal{R}_{-\alpha}(v)\cdot\nabla \omega+c_{\alpha}|\omega|^2\leq 0, \quad c_{\alpha}=\frac{[\sin\alpha]}{2}.
\een 
We now consider the following ODE:
\eqn\label{eq;ode}
y'+c_{\alpha}y^2=0 \quad \text{with}\quad y_0=\|\omega_0\|_{L^\infty}.
\een 
Then, $y(t)$ is given by 
\[
y(t)= \frac{\|\omega_0\|_{L^\infty}}{(1+c_{\alpha}\|\omega_0\|_{L^\infty}t)}.
\]
Subtracting \eqref{eq;ode} from \eqref{eq;lemvor}, we get
\eqn \label{eq:4.6d}
(\omega-y(t))_t-\Delta(\omega-y(t))+\mathcal{R}_{-\alpha}(v)\cdot\nabla (\omega-y(t))+c_\alpha(|\omega|^2-y^2(t))\leq 0.
\een
We now take the inner product of this inequality with $(\omega-y(t))^{+}$. We first note that the last term on the left-hand side of (\ref{eq:4.6d}) gives a non-negative term:
\[
\left(|\omega|^2-y^2(t))(\omega-y(t)\right)^{+}=(\omega+y(t))(\omega-y(t))(\omega-y(t))^{+} \ge 0,
\]
where we use the sign condition of $y\ge 0$. By applying Lemma {\ref{L2 lemma plus}} and Lemma {\ref{L2 lemma grad plus}} with $f=\omega-y(t)$,   
\[
\begin{split}
\frac{1}{2} \frac{d}{dt}\left\|(\omega-y(t))^{+}\right\|^2_{L^2} +\left\|\nabla(\omega-y(t))^{+}\right\|^2_{L^2}&\leq -\int \mathcal{R}_{-\alpha}(v)\cdot\nabla (\omega-y(t))(\omega-y(t))^{+}\\
& = \int \frac{[\sin\alpha]}{2}\omega \left|(\omega-y(t))^{+}\right|^2 \leq \int \frac{[\sin\alpha]}{2}\omega^{+} \left|(\omega-y(t))^{+}\right|^2\\
&\leq \frac{ \left\|\omega^{+}\right\|_{L^\infty}}{2} \left\| (\omega-y(t))^{+}\right\|^2_{L^2}.
\end{split}
\]
By Gr\"onwall's inequality and the condition $\omega_{0}\leq y_0=\|\omega_0\|_{L^\infty}$, we have
\[
\left\|(\omega(t)-y(t))^{+}\right\|^2_{L^2} \leq\exp\left[ \int^t_0\left\|\omega^{+}(\tau)\right\|_{L^\infty}d\tau\right]\left\|(\omega-y)^{+}_{0}\right\|^2_{L^2}=0.
\]
So, we conclude that $\left(\omega(x,t)-y(t)\right)^{+}=0$ for all $(x,t)$ and thus complete the proof of Lemma \ref{Decay lemma of omega}. 
\end{proof}

The first result concerning (\ref{Vis Eq of v}) is as follows.

\begin{theorem}\label{GWP theorem}\upshape
Let $v_{0}\in H^{3}$ with $\dv v_{0}=0$. Then, there exists a unique solution $v\in C([0,\infty);H^3)$ of (\ref{Vis Eq of v}) satisfying the following bound for all $t>0$:
\eqn \label{Global uniform bound}
\begin{split}
\left\|v(t)\right\|^{2}_{H^{3}}+\int^{t}_{0}\left\|\nabla v(\tau)\right\|^{2}_{H^{3}}d\tau &\leq C\left\|\omega_{0}\right\|^{2}_{H^{2}}\exp\left[C\|v_0\|^2_{L^2}\left(1+\frac{[\sin\alpha]}{2}\left\|\omega^{+}_0\right\|_{L^\infty}t\right)^2\right]\\
&\quad +\|v_0\|^2_{L^2}\left(1+\frac{[\sin\alpha]}{2}\left\|\omega^{+}_0\right\|_{L^\infty}t\right)^2.
\end{split}
\een
\end{theorem}

\begin{proof}
Given the local-in-time result in Theorem \ref{LWP theorem}, it remains to show that the maximal existence time $T^{\ast}$ is infinite. We proceed by contradiction, assuming that $T^{\ast} < \infty$. Then, by establishing a uniform bound for $\|v(t)\|_{H^3}$ on $[0, T^{\ast})$, we show that the solution can be extended to a larger interval $[0, T^{\ast} + \epsilon)$. This contradicts the maximality of $T^{\ast}<\infty$; hence, we must have $T^{\ast} = \infty$.

\vspace{1ex}

\noindent
$\blacktriangleright$ We now establish a uniform bound for $\|v(t)\|_{H^3}$ on $[0, T^{\ast})$. We begin with  the $L^{2}$ bound of $v$: 
\eqn \label{L2 bound of u with vorticity}
\frac{1}{2}\frac{d}{dt}\|v\|^2_{L^2}+\|\nabla v\|^2_{L^2}=\frac{1}{2}\int [\sin\alpha]\omega|v|^2\leq \frac{1}{2}[\sin\alpha] \|\omega^+\|_{L^\infty}\| v\|^2_{L^2}.
\een
By Gro\"nwall's inequality and (\ref{Decay of positive vorticity}), we obtain
\begin{equation}\label{Viscous L2 bound of v}
\begin{split}
\|v(t)\|^2_{L^2}+\int^t_0\|\nabla v(\tau)\|^2_{L^2}d\tau &\leq \|v_0\|^2_{L^2}\exp\left[{[\sin\alpha]} \left\|\omega^{+}_0\right\|_{L^\infty} \int^{t}_{0}\frac{1}{1+\frac{[\sin\alpha]}{2}\left\|\omega^{+}_0\right\|_{L^\infty}\tau}d\tau\right]\\
&=\|v_0\|^2_{L^2}\exp\left[2 \ln \left(1+\frac{[\sin\alpha]}{2}\left\|\omega^{+}_0\right\|_{L^\infty}t\right)\right]\\
& = \|v_0\|^2_{L^2}\left(1+\frac{[\sin\alpha]}{2}\left\|\omega^{+}_0\right\|_{L^\infty}t\right)^2 :=\mathcal{I}(t).
\end{split}
\end{equation}

To bound the derivatives of $v$, we use the equation of the vorticity. By (\ref{Eq of omega}) and (\ref{div curl}), we have 
\[
\begin{split}
\frac{1}{2}\frac{d}{dt}\left\|\omega\right\|^{2}_{L^{2}}+\left\|\nabla \omega\right\|^{2}_{L^{2}}&=-\int \left(\mathcal{R}_{-\alpha}(v)\cdot \nabla \omega\right)\cdot \omega-\sin\alpha\int |\nabla v|^2 \omega =\sin\alpha \left[\frac{1}{2}\int \omega|\omega|^2-\int |\nabla v|^2 \omega\right]\\
&\leq C \left\|\omega\right\|^{2}_{L^{4}}\left\|\omega\right\|_{L^{2}} \leq C \left\|\omega\right\|^2_{L^{2}}\left\|\nabla \omega\right\|_{L^{2}} \leq C  \left\|\omega\right\|^4_{L^{2}}+\frac{1}{2}\left\|\nabla \omega\right\|^2_{L^{2}}
\end{split}
\]
and thus we obtain 
\[
\frac{d}{dt}\left\|\omega\right\|^{2}_{L^{2}}+\left\|\nabla \omega\right\|^{2}_{L^{2}}\leq C  \left\|\omega\right\|^4_{L^{2}}.
\]
By Gr\"onwall's inequality and (\ref{Viscous L2 bound of v}), we obtain
\eqn \label{Viscous L2 bound of omega}
\left\|\omega(t)\right\|^{2}_{L^{2}}+\int^{t}_{0}\left\|\nabla \omega(s)\right\|^{2}_{L^{2}}ds\leq \left\|\omega_{0}\right\|^2_{L^{2}}\exp\left[C\mathcal{I}(t)\right].
\een

Similarly, we have 
\[
\begin{split}
\frac{1}{2}\frac{d}{dt}\left\|\nabla \omega\right\|^{2}_{L^{2}}+\left\|\Delta \omega\right\|^{2}_{L^{2}}\leq C \left\|\omega\right\|_{L^{2}} \left\|\nabla \omega\right\|^{2}_{L^{4}} \leq C \left\|\omega\right\|_{L^{2}} \left\|\nabla \omega\right\|_{L^{2}}\left\|\Delta \omega\right\|_{L^{2}}
\end{split}
\]
and thus we obtain 
\[
\frac{d}{dt}\left\|\nabla \omega\right\|^{2}_{L^{2}}+\left\|\Delta  \omega\right\|^{2}_{L^{2}}\leq C\left\|\omega\right\|^{2}_{L^{2}}\left\|\nabla \omega\right\|^{2}_{L^{2}}.
\]
By Gr\"onwall's inequality and (\ref{Viscous L2 bound of v}), we get
\eqn \label{Viscous L2 bound of grad of omega}
\left\|\nabla \omega(t)\right\|^{2}_{L^{2}}+\int^{t}_{0}\left\|\Delta\omega(s)\right\|^{2}_{L^{2}}ds\leq \left\|\nabla \omega_{0}\right\|^2_{L^{2}}\exp\left[C\mathcal{I}(t)\right].
\een

 We now bound $\Delta \omega$:
\[
\begin{split}
\frac{1}{2}\frac{d}{dt}\left\|\Delta \omega\right\|^{2}_{L^{2}}&+\left\|\nabla \Delta \omega\right\|^{2}_{L^{2}}=-\int \Delta \left(\mathcal{R}_{-\alpha}(v)\cdot \nabla \omega\right) \Delta \omega - [\sin\alpha]\int \Delta |\nabla v|^2  \Delta \omega\\
&=-\int \left(\Delta \mathcal{R}_{-\alpha}(v)\cdot \nabla \omega\right)\Delta \omega-2\sum^{2}_{k=1}\int \left(\partial_{k}\mathcal{R}_{-\alpha}(v)\cdot \nabla \partial_{k}\omega\right)\Delta\omega \\
&\quad  -\int \left(\mathcal{R}_{-\alpha}(v)\cdot \nabla \Delta \omega\right)\Delta\omega-2[\sin\alpha] \int \left(\nabla \Delta v\cdot \nabla v\right)\Delta\omega -2[\sin\alpha ]\sum^{2}_{k=1} \int\left|\partial_{k}\nabla v\right|^2\Delta \omega  \\
&\leq C\left\|\omega\right\|^2_{L^{2}}\left\|\Delta \omega\right\|^2_{L^{2}}+\frac{1}{2}\left\|\nabla \Delta \omega\right\|^2_{L^{2}},
\end{split}
\]
after some tedious but standard estimates using (\ref{L4 bound}), (\ref{Interpolation}), and (\ref{div curl}). So, we obtain
\[
\frac{d}{dt}\left\|\Delta \omega\right\|^{2}_{L^{2}}+\left\|\nabla \Delta \omega\right\|^{2}_{L^{2}}\leq C \left\|\omega\right\|^{2}_{L^{2}}\left\|\Delta \omega\right\|^{2}_{L^{2}}.
\]
By Gr\"onwall's inequality and (\ref{Viscous L2 bound of v}), we get
\eqn \label{Viscous L2 bound of Delta of omega}
\left\|\Delta \omega(t)\right\|^{2}_{L^{2}}+\int^{t}_{0}\left\|\nabla \Delta\omega(s)\right\|^{2}_{L^{2}}ds\leq \left\|\Delta\omega_{0}\right\|^2_{L^{2}}\exp\left[C\mathcal{I}(t)\right].
\een

By (\ref{Viscous L2 bound of v}), (\ref{Viscous L2 bound of omega}), (\ref{Viscous L2 bound of grad of omega}), and (\ref{Viscous L2 bound of Delta of omega}), we obtain the following bound on $[0, T^{\ast})$:
\eqn \label{Viscous H3 bound of v}
\left\|v(t)\right\|^{2}_{H^{3}}+\int^{t}_{0}\left\|\nabla v(\tau)\right\|^{2}_{H^{3}}d\tau \leq C\left\|\omega_{0}\right\|^{2}_{H^{2}}\exp\left[C\mathcal{I}(t)\right]+\mathcal{I}(t). 
\een

\noindent
$\blacktriangleright$ To finish the proof, we must rule out the possibility that the vorticity $\omega$ diverges to $-\infty$ in finite-time. Suppose that $\omega(t,x) \to -\infty$ as $t \to T^{\ast} <\infty$. On the other hand, the estimate (\ref{Viscous L2 bound of v}) remains valid on $[0, T^{\ast}-\epsilon]$ for any $\epsilon> 0$ and thus we have (\ref{Viscous H3 bound of v}) on $[0, T^{\ast}-\epsilon]$ as well. This also implies 
\[
\left\|\omega(t)\right\|_{L^{\infty}} \leq C\left\|v(t)\right\|_{H^3} \leq C\left\|\omega_{0}\right\|_{H^{3}}\exp\left[C\mathcal{I}(t)\right]+C\mathcal{I}(t)
\]
for all $t\in[0, T^{\ast}-\epsilon]$. This contradicts the assumption that $\omega(t,x) \to -\infty$ as $t \to T^{\ast} < \infty$ because $\mathcal{I}(t)$ is bounded for all $t>0$ and thus the solution cannot blow up in finite time.
\end{proof}

%%%%%%%%%%%%%%%%%%%%%%%%%%%%%%%%%%%%%%5
\subsection{Long-time behavior of the difference between two solutions}
%%%%%%%%%%%%%%%%%%%%%%%%%%%%%%%%%%%%%%%5
In \cite{Ohkitani}, by tracing the evolution of vorticity for different values of $\alpha$, Ohkitani establishes a continuous transition that links the two systems, showing that the qualitative features of the flow are insensitive to small perturbations in the parameters. To verify this approach, we investigate the long-time behaviors of the two different velocities according to the angle parameter $\alpha$. Let $(v_1,\alpha_1)$ and $(v_2,\alpha_2)$ be two solutions of each below equations:
\[
\begin{split}
\partial_t v_1+\mathbb{P}(\mathcal{R}_{-\alpha_{1}}(v_1)\cdot \nabla v_1)&=\Delta v_1, \quad v_{1}(0,x)=v_{(1)}(x)\\
 \partial_t v_2+\mathbb{P}(\mathcal{R}_{-\alpha_{2}}(v_2)\cdot \nabla v_2)&=\Delta v_2, \quad v_{2}(0,x)=v_{(2)}(x)
\end{split}
\]
with $\dv v_{1}= \dv v_{2}=0$. Let $\widetilde{v}(t,x)=v_1(t,x)-v_2(t,x)$. Then, $\widetilde{v}$ satisfies
\eqn \label{eq;difvor}
\partial_t \widetilde{v}+\mathbb{P}[(\mathcal{R}_{-\alpha_{1}}-\mathcal{R}_{-\alpha_{2}})(v_1)\cdot \nabla v_1]+\mathbb{P}[\mathcal{R}_{-\alpha_{2}}(\widetilde{v})\cdot \nabla v_1]+\mathbb{P}[\mathcal{R}_{-\alpha_{2}}(v_{2})\cdot \nabla \widetilde{v}]=\Delta \widetilde{v}
\een
with $\widetilde{v}_{0}(x)=v_{(1)}(x)-v_{(2)}(x)$. Since $v_{1}$ and $v_{2}$ satisfy (\ref{Global uniform bound}), we now define the following quantity: 
\[
\begin{split}
\mathcal{J}^{(i)}(t)&=\left\|\omega_{(i)}\right\|^{2}_{H^{2}}\exp\left[C\|v_{(i)}\|^2_{L^2}\left(1+\frac{[\sin\alpha_{i}]}{2}\left\|\omega^{+}_{(i)}\right\|_{L^\infty}t\right)^2\right] \\
&+\|v_{(i)}\|^2_{L^2}\exp\left[C\left(1+\frac{[\sin\alpha_{i}]}{2}\left\|\omega^{+}_{(i)}\right\|_{L^\infty}t\right)^2\right], \quad i=1,2,
\end{split}
\]
where $\omega_{(i)}$ denotes the vorticity associated with the velocity field $v_{(i)}$, $i=1,2$.

\begin{theorem}\upshape \label{Theorem Difference bound}
$\widetilde{v}$ satisfies the following bound for all $t>0$:
\eqn \label{Global uniform bound New}
\left\|\widetilde{v}(t)\right\|^{2}_{H^{3}}+\int^{t}_{0}\left\|\nabla \widetilde{v}(\tau)\right\|^{2}_{H^{3}}d\tau \leq \left[\left\|\widetilde{v}_{0}\right\|^2_{H^{3}}+C_{\alpha}\,t\, \mathcal{J}^{(1)}(t)\right]\exp\left[C\left(\mathcal{J}^{(1)}(t)+\mathcal{J}^{(2)}(t)\right)\right].
\een
where $C_{\alpha}=C\max\{|\cos\alpha_1-\cos\alpha_2|,|\sin\alpha_1-\sin\alpha_2|\}$.
\end{theorem}

\begin{proof}
We apply the derivative $D^k$, $k=0,1,2,3$, to (\ref{eq;difvor}) and we take $L^2$ inner product of the resulting equation with $D^{k}v$.  Then, we have 
\begin{equation*}
\begin{split}
\frac{1}{2} \frac{d}{dt}\|\widetilde{v}\|^2_{H^3}+\|\nabla \widetilde{v}\|^2_{H^3}=&-\sum^{3}_{k=0}\int D^{k}[(\mathcal{R}_{-\alpha_{1}}-\mathcal{R}_{-\alpha_{2}})(v_1)\cdot \nabla v_1]\cdot D^{k}\widetilde{v}\\
&-\sum^{3}_{k=0}\int D^{k}[\mathcal{R}_{-\alpha_{2}}(\widetilde{v})\cdot \nabla v_1]\cdot D^{k}\widetilde{v}-\sum^{3}_{k=0}\int D^{k}[\mathcal{R}_{-\alpha_{2}}(v_{2})\cdot \nabla \widetilde{v}]\cdot D^{k}\widetilde{v}\\
 =& \text{(I)+(II)+(III)}. 
\end{split} 
\end{equation*}
By (\ref{Product estimates}), we first bound $\text{(I)}$:
\[
\begin{split}
\text{(I)} &\leq C \left\|(\mathcal{R}_{-\alpha_{1}}-\mathcal{R}_{-\alpha_{2}})(v_1)\right\|_{H^{3}} \left\|\nabla v_1\right\|_{H^{3}} \left\|\widetilde{v}\right\|_{H^{3}} \leq C \left\|\nabla v_1\right\|^2_{H^{3}} \left\|\widetilde{v}\right\|^2_{H^{3}} +C\left\|(\mathcal{R}_{-\alpha_{1}}-\mathcal{R}_{-\alpha_{2}})(v_1)\right\|^2_{H^{3}} \\
&\leq C \left\|\nabla v_1\right\|^2_{H^{3}} \left\|\widetilde{v}\right\|^2_{H^{3}} +C_{\alpha}\left\|v_1\right\|^2_{H^{3}}.
\end{split}
\]
We now estimate $\text{(II)}$ using  (\ref{Product estimates 2}):
\[
\begin{split}
\text{(II)}&\leq C\sum^{3}_{k=0}\left(\left\|\nabla v_{1}\right\|_{L^{2}}\left\|D^{k}\mathcal{R}_{-\alpha_{2}}(\widetilde{v})\right\|_{L^{4}} \left\|D^{k}\widetilde{v}\right\|_{L^{4}}+\left\|D^{k+1}v_{1}\right\|_{L^{2}} \left\|\mathcal{R}_{-\alpha_{2}}(\widetilde{v})\right\|_{L^{4}} \left\|D^{k}\widetilde{v}\right\|_{L^{4}}\right)\\
&\leq C\sum^{3}_{k=0}\left(\left\|\nabla v_{1}\right\|_{L^{2}}\left\|D^{k}\widetilde{v}\right\|^2_{L^{4}}+\left\|D^{k+1}v_{1}\right\|_{L^{2}} \left\|\widetilde{v}\right\|^2_{L^{4}} +\left\|D^{k+1}v_{1}\right\|_{L^{2}} \left\|D^{k}\widetilde{v}\right\|^2_{L^{4}}\right)\\
&\leq C\left\|\nabla v_{1}\right\|^2_{H^{3}}\left\|\widetilde{v}\right\|^2_{H^{3}}+\frac{1}{4}\left\|\nabla \widetilde{v}\right\|^2_{H^{3}}.
\end{split}
\]
We finally bound $\text{(III)}$ using Lemma \ref{Commutator lemma} and (\ref{div curl}):
\[
\begin{split}
\text{(III)}&=-\sum^{3}_{k=0}\int [\mathcal{R}_{-\alpha_{2}}(v_{2})\cdot \nabla D^{k}\widetilde{v}]\cdot D^{k}\widetilde{v}-\sum^{3}_{k=1}\int \left\{[D^{k}, \mathcal{R}_{-\alpha_{2}}(v_{2})]\cdot \nabla \widetilde{v}\right\}\cdot D^{k}\widetilde{v}\\
&=\frac{[\sin\alpha_{2}]}{2}\sum^{3}_{k=0}\int\omega_{2}\left|D^{k}\widetilde{v}\right|^2-\sum^{3}_{k=1}\int \left\{[D^{k}, \mathcal{R}_{-\alpha_{2}}(v_{2})]\cdot \nabla \widetilde{v}\right\}\cdot D^{k}\widetilde{v}\\
&\leq \left\|\omega_{2}\right\|_{L^{2}}\sum^{3}_{k=0}\left\|D^{k}\widetilde{v}\right\|^2_{L^{4}} +C\sum^{3}_{k=1}\left\|D^{k}\widetilde{v}\right\|_{L^{2}}\left(\left\|\nabla \mathcal{R}_{-\alpha_{2}}(v_{2})\right\|_{L^{\infty}}\left\|D^{k}\widetilde{v}\right\|_{L^{2}}+\left\|D^{k}\mathcal{R}_{-\alpha_{2}}(v_{2})\right\|_{L^{2}}\left\|\nabla \widetilde{v}\right\|_{L^{\infty}}\right)\\
&\leq C \left\|\omega_{2}\right\|_{L^{2}}\left\|\widetilde{v}\right\|_{H^{3}}\left\|\nabla \widetilde{v}\right\|_{H^{3}}+C\left\|\nabla \mathcal{R}_{-\alpha_{2}}(v_{2})\right\|_{L^{\infty}}\left\|\widetilde{v}\right\|^2_{H^{3}}+C \left\|\nabla \mathcal{R}_{-\alpha_{2}}(v_{2})\right\|_{H^{3}}\left\|\widetilde{v}\right\|_{H^{3}}\left\|\nabla \widetilde{v}\right\|_{H^{3}}\\
&\leq C \left\|\nabla v_{2}\right\|^2_{H^{3}}\left\|\widetilde{v}\right\|^2_{H^{3}}+\frac{1}{4}\left\|\nabla \widetilde{v}\right\|^2_{H^{3}}.
\end{split}
\]
From these bounds, we have
\[
\frac{d}{dt}\|\widetilde{v}\|^2_{H^3}+\|\nabla \widetilde{v}\|^2_{H^3}\leq C\left(\left\|\nabla v_1\right\|^2_{H^{3}}+\left\|\nabla v_2\right\|^2_{H^{3}} \right)\|\widetilde{v}\|^2_{H^3} +C_{\alpha}\left\|v_1\right\|^2_{H^{3}}.
\]
By Gr\"ownwall's inequality, we obtain (\ref{Global uniform bound New}) for all $t>0$.
\end{proof}

\begin{remark}\upshape \label{Remark 4.1}
If $\omega_{0}\leq 0$, (\ref{L2 bound of u with vorticity}) is replaced with
\[
\frac{1}{2}\frac{d}{dt}\|v\|^2_{L^2}+\|\nabla v\|^2_{L^2}=\frac{1}{2}\int [\sin\alpha]\omega|v|^2\leq 0
\]
because $\omega\leq 0$ by Lemma  \ref{Max Principle Lemma} below. From this, it can be shown that 
\[%eqn \label{Uniform bounds with negative vorticity New}
\left\|v(t)\right\|^{2}_{H^{3}}+\int^{t}_{0}\left\|\nabla v(\tau)\right\|^{2}_{H^{3}}d\tau
\leq \left\|v_{0}\right\|^{2}_{L^{2}}+C\left\|\omega_{0}\right\|^{2}_{H^{2}}\exp\left[C\left\|v_{0}\right\|^{2}_{L^{2}}\right].
\]%een
So, the time dependency on the right-hand side of (\ref{Global uniform bound New}) is eliminated and  (\ref{Global uniform bound New}) can be reduced to 
\[
\left\|\widetilde{v}(t)\right\|^{2}_{H^{3}}+\int^{t}_{0}\left\|\nabla \widetilde{v}(\tau)\right\|^{2}_{H^{3}}d\tau \leq \left[\left\|\widetilde{v}_{0}\right\|^2_{H^{3}}+CC_{\alpha}\, t\,  \mathcal{I}_{(1)}\right]\exp\left[C\left(\mathcal{I}_{(1)}+\mathcal{I}_{(2)}\right)\right]
\]
for all $t>0$ where 
\[
\mathcal{I}_{(1)}=\left\|v_{(1)}\right\|^{2}_{L^{2}}+C\left\|\omega_{(1)}\right\|^{2}_{H^{2}}\exp\left[C\left\|v_{(1)}\right\|^{2}_{L^{2}}\right], \quad \mathcal{I}_{(2)}=\left\|v_{(2)}\right\|^{2}_{L^{2}}+C\left\|\omega_{(2)}\right\|^{2}_{H^{2}}\exp\left[C\left\|v_{(2)}\right\|^{2}_{L^{2}}\right].
\]
\end{remark}

%%%%%%%%%%%%%%%%%%%%%%%%%
\section{Temporal decay rates}\label{sec:5}
%%%%%%%%%%%%%%%%%%%%%%%%
Since the solution in Theorem \ref{GWP theorem} exists globally-in-time,  the natural question is twofold: do the solutions decay as $t\rightarrow \infty$, and what is the rate of this decay?  It is  well-known that the long time behavior of solutions to dissipative equations can be controlled by imposing some conditions on the initial data. Motivated by \cite{Guo Wang, Schonbek 1986}, we investigate the temporal decay rates of (\ref{Vis Eq of v}). Before stating our result, we first give the maximum principle for parabolic equations.

\begin{lemma}\upshape \label{Max Principle Lemma}
Suppose that $f:[0,T)\times \mathbb{R}^{2} \rightarrow \mathbb{R}$ be a function in $C([0,T);H^2)\cap C^1([0,T);L^2) $ with $f_0\in H^2$ and let $v(t,x)$ be a vector field on $[0,T)\times \mathbb{R}^{2}$ which belongs to $L^1(0,T;W^{1,\infty})$. Suppose
\begin{equation}\label{maxprin}
f_{t}+v\cdot \nabla f-\Delta f\leq 0 
\end{equation}
for a.e. $(x,t)$ and $f_{0}(x)\leq 0$ for all $x\in\mathbb{R}^{2}$. Then, 
\[
f(t,x)\leq 0 \quad \text{for all $t\in [0,T)$ and $x\in\mathbb{R}^{2}$}.
\]
\end{lemma}

\begin{proof}
We now take the inner product of (\ref{maxprin}) with $f^+$. By Lemma {\ref{L2 lemma plus}} and Lemma {\ref{L2 lemma grad plus}}, 
\[
\frac{1}{2} \frac{d}{dt}\left\|f^+\right\|^2_{L^2} +\left\|\nabla f^{+}\right\|^2_{L^2}\leq -\int (v\cdot\nabla f) f^+ = \int \frac{1}{2}(\dv v)|f^{+}|^2\leq \frac{ \left\|\dv v\right\|_{L^\infty}}{2} \left\| f^{+}\right\|^2_{L^2}.
\]
By Gr\"onwall's inequality with $f^{+}_{0}=0$,  we have
\[
\left\|f^{+}\right\|^2_{L^2} \leq\exp\left[ \int^t_0\left\|\dv v(\tau)\right\|_{L^\infty}d\tau\right]\left\|f^{+}_{0}\right\|^2_{L^2}=0.
\]
So, we conclude that $f^{+}=0$ for all $(x,t)$ and thus complete the proof of Lemma \ref{Max Principle Lemma}. 
\end{proof}

We also rewrite $\mathcal{R}_{-\alpha}(v)\cdot \nabla v$ as follows. From (\ref{Rotation operator}),
\[
\begin{split}
\mathcal{R}_{-\alpha}(v)\cdot \nabla v=\begin{pmatrix}
[\cos \alpha] v_{1}\partial_{1}v_{1} -[\sin\alpha]v_{2}\partial_{2}v_{2} -[\sin\alpha] v_{1}\partial_{2}v_{1} +[\cos\alpha]v_{2}\partial_{2}v_{1}\\
[\cos \alpha] v_{1}\partial_{1}v_{2} +[\sin\alpha]v_{2}\partial_{1}v_{2} +[\sin\alpha] v_{1}\partial_{1}v_{1} +[\cos\alpha]v_{2}\partial_{2}v_{2}
\end{pmatrix}.
\end{split}
\]
By using $\dv v=0$, we derive 
\[
v_{2}\partial_{2}v_{1}=\partial_{2}(v_{1}v_{2})-v_{1}\partial_{2}v_{2}=\partial_{2}(v_{1}v_{2})+v_{1}\partial_{1}v_{1}, \quad v_{1}\partial_{1}v_{2}=\partial_{1}(v_{1}v_{2})-v_{2}\partial_{1}v_{1}=\partial_{1}(v_{1}v_{2})+v_{2}\partial_{2}v_{2}
\]
from which we obtain 
\eqn\label{div form of advection}
\mathcal{R}_{-\alpha}(v)\cdot \nabla v=\begin{pmatrix}
[\cos \alpha] \partial_{1}\left(v^{2}_{1}\right) -\frac{[\sin\alpha]}{2}\partial_{2}\left(v^2_{2}\right) -\frac{[\sin\alpha]}{2} \partial_{2}\left(v^2_{1}\right) +[\cos\alpha]\partial_{2}(v_{1}v_{2})\\
[\cos\alpha]\partial_{2}\left(v^2_{2}\right)+\frac{[\sin\alpha]}{2}\partial_{1}\left(v^2_{2}\right) +\frac{[\sin\alpha]}{2} \partial_{1}\left(v^2_{1}\right) +[\cos\alpha]\partial_{1}(v_{1}v_{2})
\end{pmatrix}.
\een
So, each term in $\mathcal{R}_{-\alpha}(v)\cdot \nabla v$ can be expressed as a linear combination of the derivatives of the products $v_{i}v_{j}$, $i,j=1,2$. This expression is essential to derive temporal decay rates of $v$. In fact, if we obtain an energy inequality of the form
\eqn \label{Energy inequality}
\frac{1}{2}\frac{d}{dt}\|v\|^2_{L^2}+\|\nabla v\|^2_{L^2}\leq 0,
\een
 we can follow \cite[Section 3]{Niche Schonbek} to derive 
\[%eqn \label{L2 decay of v negative omega}
\|v(t)\|_{L^2}\leq \frac{C_{0}}{\sqrt{1+t}},
\]%een
where $C_{0}$ only depends on $v_{0}\in L^{2}\cap L^{1}$. However, the assumption $v_0 \in L^2 \cap L^1$ is insufficient here because the energy equality involves the vorticity
\[
\frac{1}{2}\frac{d}{dt}\|v\|^2_{L^2}+\|\nabla v\|^2_{L^2}=\frac{1}{2}\int [\sin\alpha]\omega|v|^2.
\]
Choosing a nonpositive initial vorticity, Lemma \ref{Max Principle Lemma} implies that $\omega(t,x)\leq0$, which ensures (\ref{Energy inequality}).  This argument can be summarized as follows.

\begin{theorem} \label{L2 Decay theorem}\upshape
Let $v_0\in H^3\cap L^1$ with $\dv v_{0}=0$ and $\omega_0\leq0$. Then, $v$ decays in time as follows:
\eqn \label{L2 decay of v negative omega}
\|v(t)\|_{L^2}\leq \frac{C_{0}}{\sqrt{1+t}},
\een
where $C_{0}$ only depends on $v_{0}\in L^{2}\cap L^{1}$.
\end{theorem}

Using Theorem \ref{L2 Decay theorem}, we can derive the decay rates of the derivatives of $v$.

 \begin{corollary} \upshape \label{Higher order decay}
 Let $v_0\in H^3\cap L^1$ with $\dv v_{0}=0$ and $\omega_0\leq0$. Then, there exists $T_{0}$ depending on $v_{0}\in L^{2}\cap L^{1}$ such that $v$ decays in time as follows: 
\[
\left\|v(t)\right\|_{\dot{H}^k}\leq \frac{C_{0}}{(1+t)^\frac{k+1}{2}}, \quad k=1,2,3
\]
for all $t>T_{0}$, where $C_{0}$ only depends on $v_0\in H^3\cap L^1$.
\end{corollary}

\begin{proof}
We apply the derivative $D^k$ to (\ref{Vis Eq of v}) and we take $L^2$ inner product of the resulting equation with $D^{k}v$.  By Lemma \ref{Commutator lemma}, (\ref{L4 bound}) and (\ref{Interpolation}), we have 
\begin{equation*}
\begin{split}
\frac{1}{2} \frac{d}{dt}\|v\|^2_{\dot{H}^k}+\|\nabla v\|^2_{\dot{H}^k} &=-\int \left(\left[D^{k}, \mathcal{R}_{-\alpha}(v)\right]\cdot \nabla v\right)\cdot D^{k}v -\int \left(\mathcal{R}_{-\alpha}(v)\cdot \nabla D^{k}v\right) \cdot D^{k}v\\
& \leq C\left\|\nabla v\right\|_{L^2} \left\|v\right\|_{\dot{H}^k} \left\|\nabla v\right\|_{\dot{H}^k}\leq C\left\|v\right\|_{L^2} \left\|\nabla v\right\|^2_{\dot{H}^k}.
\end{split}
\end{equation*}
By (\ref{L2 decay of v negative omega}), for a sufficiently large time $T_{0}>0$ satisfying $C\left\|\nabla v\right\|_{L^2}\leq \frac{1}{2}$, we have 
\[
 \frac{d}{dt}\|v\|^2_{\dot{H}^k}+\|\nabla v\|^2_{\dot{H}^k}\leq 0
\]
for all $t\ge T_{0}$. Using (\ref{Interpolation}) again, we obtain 
\[
\frac{d}{dt}\|v\|^2_{\dot{H^k}}+C_{0}(1+t)^\frac{1}{k}\|v\|^{2\frac{k+1}{k}}_{\dot{H}^k}\leq 0.
\]
Solving this inequality yields
\[
\left\|v(t)\right\|_{\dot{H}^k}\leq \frac{C_{0}}{(1+t)^\frac{k+1}{2}}
\]
for all $t\ge T_{0}$. This completes the proof of Corollary \ref{Higher order decay}.
\end{proof}

Motivated by \cite{Guo Wang}, we now improve Theorem \ref{L2 Decay theorem} by replacing the $L^1$ assumption with the $\dot{H}^{-\sigma}$ condition for $0 < \sigma < 2$. To address Sobolev spaces with negative regularity, we recall the Hardy-Littlewood-Sobolev inequality \cite{Stein} in the two-dimensional case: for $0<\sigma<1$, we have 
\eqn \label{HLS}
\left\|f\right\|_{\dot{H}^{-\sigma}}\leq C \left\|f\right\|_{L^{p}}, \quad \frac{1}{2}+\frac{\sigma}{2}=\frac{1}{p}.
\een

We now extend the Fourier splitting method \cite{Schonbek 1986} to the setting of Sobolev spaces with negative regularity.

\begin{lemma}\upshape \label{FS Lemma}
Let $\sigma>0$. Suppose we have 
\[
\frac{d}{dt} \left\|f\right\|^{2}_{L^{2}} +\left\|\nabla f\right\|^{2}_{L^{2}}\leq 0, \quad \sup_{t>0}\|f(t)\|_{\dot{H}^{-\sigma}} \leq C_{0}.
\]
Then, $v$ decays in time as follows:
\eqn \label{Decay by FS}
\|f(t)\|_{L^{2}} \leq C_1(1+t)^{-\frac{\sigma}{2}},
\een
where $C_1$ depends only on $C_0$.
\end{lemma}

\begin{proof}
Using the Plancherel's theorem, 
\[
\begin{aligned}
\frac{d}{dt}\|f(t)\|_{L^{2}}^2 & \leq - \int |\xi|^{2} \left|\widehat f(t,\xi)\right|^2  d\xi \leq -\int_{\{|\xi|^{2} > \frac{m}{1+t} \}} |\xi|^{2} \left|\widehat f(t,\xi)\right|^2 d\xi \\
&\leq -\frac{m}{1+t} \int_{\{|\xi|^{2} > \frac{m}{1+t}\}}  \left|\widehat f(t,\xi)\right|^2 d\xi = -\frac{m}{1+t}\|f(t)\|_{L^{2}}^2 + \frac{m}{1+t}\int_{\{|\xi|^{2} \le \frac{m}{1+t}\}}  \left|\widehat f(t,\xi)\right|^2  d\xi,
\end{aligned}
\]
where $m>0$ will be determined below. From this, we obtain 
\[
\frac{d}{dt}\|f(t)\|_{L^{2}}^2 +\frac{m}{1+t}\|f(t)\|_{L^{2}}^2 \leq   \frac{m}{1+t}\int_{\{|\xi|^{2} \le \frac{m}{1+t}\}}  \left|\widehat f(t,\xi)\right|^2  d\xi
\]
which implies 
\[
\begin{split}
\frac{d}{dt}\left[(1+t)^m\|f(t)\|_{L^{2}}^2  \right] &\leq m(1+t)^{m-1}\int_{\{|\xi|^{2} \le \frac{m}{1+t}\}} |\xi|^{2\sigma} |\xi|^{-2\sigma}\left|\widehat f(t,\xi)\right|^2  d\xi\\
&\leq C(1+t)^{m-1-\sigma} \int_{\{|\xi|^{2} \le \frac{m}{1+t}\}}  |\xi|^{-2\sigma}\left|\widehat f(t,\xi)\right|^2  d\xi \\
&\leq C(1+t)^{m-1-\sigma}\sup_{t>0}\|f\|^2_{\dot{H}^{-\sigma}} \leq C_{1}(1+t)^{m-1-\sigma}.
\end{split}
\]
By choosing $\sigma<m$ and by integrating the above inequality in time, (\ref{Decay by FS}) is achieved.
\end{proof}

%We now establish the temporal decay rates of $v$ for initial data $v_0 \in H^3 \cap \dot{H}^{-\sigma}$.

\begin{theorem} \label{L2 Decay theorem 2}\upshape
Let $v_{0} \in H^{3} \cap \dot{H}^{-\sigma}$ with $\dv v_{0}=0$, $\omega_0\leq0$ and $0<\sigma<2$. Then, $v$ satisfies  the following bounds for all $t>0$:
\begin{subequations} \label{Negative space bounds}
\begin{align}
&\left\|v(t) \right\|^{2}_{L^{2}\cap\dot{H}^{-\sigma}}+ \int^{t}_{0}\left\|v(\tau) \right\|^{2}_{\dot{H}^{1} \cap \dot{H}^{1-\sigma}}d\tau \leq C_{0}, \label{Negative space bounds a}\\
& \|v(t)\|_{L^{2}} \leq C_{1}(1+t)^{-\frac{\sigma}{2}} \label{Negative space bounds b},
\end{align}
\end{subequations}
where $C_{0}$ and $C_1$ only depend on $v_{0} \in L^{2} \cap \dot{H}^{-\sigma}$.
\end{theorem}

\begin{proof}
Since 
\[
\mathcal{R}_{-\alpha}(v)\cdot \nabla v=\mathbb{D}_{\alpha}(v,v),
\]
where $\mathbb{D}_{\alpha}(v,v)$ is a linear combination of the derivatives of $v_{i}v_{j}$, $i,j=1,2$ as noted in (\ref{div form of advection}), the range of $\sigma$ is increased from $0<\sigma<1$ to $0<\sigma<2$ in Theorem \ref{L2 Decay theorem 2}. Moreover, since  
\[
\|v\|_{\dot{H}^{-\sigma_{1}}}\leq C\|v\|^{\epsilon}_{L^{2}}\|v\|^{1-\epsilon}_{\dot{H}^{-\sigma_{2}}}, \quad 0<\sigma_{1}<\sigma_{2}, \ \epsilon=1-\frac{\sigma_{1}}{\sigma_{2}} \in (0,1),
\]
we are able to use bounds of $v$ in $L^{2}\cap \dot{H}^{-\sigma'}$ with $\sigma'<\sigma$ to deal with $v\in L^{2}\cap \dot{H}^{-\sigma}$. By Lemma \ref{FS Lemma}, it is enough to derive (\ref{Negative space bounds a}) to show  (\ref{Negative space bounds b}). To this end, we multiply (\ref{Vis Eq of v}) by $\Lambda^{-2}v$ and integrate over $\mathbb{R}^2$ to obtain 
\[
\begin{split}
\frac{1}{2}\frac{d}{dt} \left\|v\right\|^{2}_{{\dot{H}^{-1}}}+ \left\|v \right\|^{2}_{L^{2}}&=-\int \Lambda^{-1}\mathbb{D}_{\alpha}(v,v)\cdot \Lambda^{-1}v \leq C\sum^{2}_{i,j=1}\left\|v_{i}v_{j}\right\|_{L^{2}}\left\|\Lambda^{-1}v \right\|_{L^2}\\
&\leq C \left\|v\right\|^2_{L^{4}}\left\|v\right\|_{{\dot{H}^{-1}}} \leq C\left\|v\right\|_{L^{2}} \left\|\nabla v\right\|_{L^{2}}\left\|v\right\|_{{\dot{H}^{-1}}} \leq C \left\|\nabla v\right\|^2_{L^2}\left\|v\right\|^{2}_{{\dot{H}^{-1}}}+\frac{1}{2} \left\|v\right\|^2_{L^2}
\end{split}
\]
and thus we derive
\[
\frac{d}{dt} \left\|v\right\|^{2}_{{\dot{H}^{-1}}}+ \left\|v \right\|^{2}_{L^{2}}\leq C \left\|\nabla v\right\|^2_{L^2}\left\|v\right\|^{2}_{{\dot{H}^{-1}}}.
\]
By  (\ref{Energy inequality}) and Gr\"ownwall's inequality, we derive the following bound:
\eqn \label{Bound of v 1 dd}
\left\|v(t)\right\|^{2}_{{\dot{H}^{-1}}}+\int^{t}_{0} \left\|v(\tau) \right\|^{2}_{L^{2}}d\tau \leq \left\|v_{0}\right\|^{2}_{{\dot{H}^{-1}}}\exp\left[C\left\|v_{0} \right\|^{2}_{L^{2}}\right]=C_{1}.
\een

We now take $1<\sigma<2$: by (\ref{HLS}),
\[
\begin{split}
\frac{1}{2}\frac{d}{dt}\left\|v\right\|^{2}_{{\dot{H}^{-\sigma}}}+\left\|v\right\|^{2}_{{\dot{H}^{1-\sigma}}}&=-\int \Lambda^{-\sigma}\mathbb{D}_{\alpha}(v,v)\cdot \Lambda^{-\sigma}v \leq C\sum^{2}_{i,j=1}\left\|\Lambda^{1-\sigma}(v_{i}v_{j})\right\|_{L^2}\left\|v\right\|_{{\dot{H}^{-\sigma}}}\\
&\leq C\sum^{2}_{i,j=1}\left\|v_{i}v_{j}\right\|_{L^\frac{2}{\sigma}}\left\|v\right\|_{{\dot{H}^{-\sigma}}}\leq C\left\|v\right\|^2_{L^\frac{4}{\sigma}}\left\|v\right\|_{{\dot{H}^{-\sigma}}}.
\end{split}
\]
Since $2<\frac{4}{\sigma}<4$, we have
\[
\left\|v\right\|^2_{L^\frac{4}{\sigma}}\leq C \left(\left\|v\right\|^2_{L^2}+\left\|\nabla v\right\|^2_{L^2}\right)
\]
by (\ref{L4 bound}) and thus we obtain 
\[
\frac{1}{2}\frac{d}{dt}\left\|v\right\|^{2}_{{\dot{H}^{-\sigma}}}+\left\|v\right\|^{2}_{{\dot{H}^{1-\sigma}}}\leq C \left(\left\|v\right\|^2_{L^2}+\left\|\nabla v\right\|^2_{L^2}\right)\left\|v\right\|^2_{{\dot{H}^{-\sigma}}}+C \left(\left\|v\right\|^2_{L^2}+\left\|\nabla v\right\|^2_{L^2}\right).
\]
So, we can use (\ref{Energy inequality}), (\ref{Bound of v 1 dd}), and Gr\"ownwall's inequality to obtain
\begin{equation} \label{Bound of v 2 dd}
\left\|v(t)\right\|^{2}_{{\dot{H}^{-\sigma}}}+\int^{t}_{0} \left\|v(\tau)\right\|^{2}_{{\dot{H}^{1-\sigma}}}d\tau \leq \left(\left\|v_{0}\right\|^{2}_{{\dot{H}^{-\sigma}}}+C\left\|v_{0}\right\|^{2}_{L^{2}}+CC_{1}\right)\exp\left[C\left(\left\|v_{0} \right\|^{2}_{L^{2}}+\mathcal{D}_{1}\right)\right]=C_{2}.
\end{equation}

By (\ref{Energy inequality}), (\ref{Bound of v 1 dd}), and (\ref{Bound of v 2 dd}), we have
\eqn \label{Total bound of v dd}
\left\|v(t) \right\|^{2}_{L^{2}\cap\dot{H}^{-\sigma}}+ \int^{t}_{0}\left\|v(\tau) \right\|^{2}_{\dot{H}^{1} \cap \dot{H}^{1-\sigma}}d\tau \leq \left\|v_{0}\right\|^{2}_{L^{2}}+C_{1}+C_{2}=C_{0}.
\een
This completes the proof of Theorem \ref{L2 Decay theorem 2}.
\end{proof}

Following the direction of Corollary \ref{Higher order decay}, we also obtain the decay rates for the derivatives of $v$.

\begin{corollary} \upshape \label{Higher order decay 2}
Let $v_{0} \in H^{3} \cap \dot{H}^{-\sigma}$ with $\dv v_{0}=0,$ $\omega_0\leq0$ and $0<\sigma<2$. Then, there exists $T_{0}$ depending on $v_{0}\in L^{2}\cap \dot{H}^{-\sigma}$ such that $v$ decays in time as follows: 
\[
\left\|v(t)\right\|_{\dot{H}^k}\leq \frac{C_{0}}{(1+t)^\frac{k+\sigma}{2}}, \quad k=1,2,3
\]
for all $t>T_{0}$, where $C_{0}$ only depends on $v_0\in H^{3} \cap \dot{H}^{-\sigma}$.
\end{corollary}

\begin{remark}\upshape[Decay rates of $v$ for initial vorticity with mixed signs]
In order to prove Theorem \ref{L2 Decay theorem 2}, the sign condition is crucially used to derive (\ref{Energy inequality}). We now do not impose the sign condition to vorticity. By (\ref{L2 bound of u with vorticity}) and Lemma \ref{Decay lemma of omega},
\[
\frac{1}{2}\frac{d}{dt}\|v\|^2_{L^2}+\|\nabla v\|^2_{L^2}\leq  \frac{[\sin\alpha]\|\omega_0\|_{L^\infty}}{(2+[\sin\alpha]\|\omega_0\|_{L^\infty}t)}\|v\|^2_{L^2}.
\]
Using this, we are not able to derive the temporal decay rate of $\|v(t)\|_{L^2}$. Therefore, certain conditions may be required to derive the temporal decay rate of $\|v(t)\|_{L^2}$. To handle this issue, we introduce two approaches.

\vspace{1ex}

\noindent
$\blacktriangleright$ Instead of (\ref{L2 bound of u with vorticity}), we can bound $v$ as follows:
\[
\frac{1}{2}\frac{d}{dt}\|v\|^2_{L^2}+\|\nabla v\|^2_{L^2}\leq C \left\|\omega\right\|_{L^{2}}\|v\|^2_{L^4}\leq C \|v\|_{L^2}\|\nabla v\|^2_{L^2}.
\]
So, if $\|v_{0}\|_{L^2}$ is sufficiently small, 
\[
\frac{d}{dt}\|v\|^2_{L^2}+\|\nabla v\|^2_{L^2}\leq 0.
\]
With this inequality in hand, we can prove Theorem \ref{L2 Decay theorem 2}.

\vspace{1ex}

\noindent
$\blacktriangleright$ We can also bound $v$ as follows: 
\eqn \label{Not good bound}
\frac{d}{dt}\|v\|^2_{L^2}+\|\nabla v\|^2_{L^2} \leq \frac{C}{(1+t)}\|\omega^+\|^\frac{1}{3} _{L^1}\|v\|^2_{L^2}.
\een
If we are able to show 
\[
\|\omega^+(t)\|_{L^1}\leq \frac{C}{(1+t)^{3\beta}}
\]
for some $\beta>0$, (\ref{Not good bound}) is replaced with  
\[
\frac{d}{dt}\|v\|^2_{L^2}+\|\nabla v\|^2_{L^2} \leq \frac{C}{(1+t)^{1+\beta}}\|v\|^2_{L^2}.
\]
By Gr\"onwall's inequality, 
\eqn \label{New L2 bound of v beta}
\|v(t)\|^2_{L^2}+\int^{t}_{0}\|\nabla v(\tau)\|^2_{L^2}d\tau \leq \|v_{0}\|^2_{L^2}\exp\left[C\right]\leq C\|v_{0}\|^2_{L^2}
\een
which in turn implies 
\[
\frac{d}{dt}\|v\|^2_{L^2}+\|\nabla v\|^2_{L^2} \leq \frac{C}{(1+t)^{1+\beta}}.
\]
By modifying Lemma \ref{FS Lemma}, 
\[
\|v(t)\|_{L^{2}} \leq  C(1+t)^{-\frac{\sigma}{2}}+C(1+t)^{-\frac{\beta}{2}} \leq C(1+t)^{-\frac{1}{2}\min\{\sigma, \beta\}}.
\]
\end{remark}

%%%%%%%%%%%%%%%
\section*{Acknowledgments}
%%%%%%%%%%%%%%%
H.B. would like to thank the Department of Mathematics at UC Riverside for their hospitality during the period this work was conducted. H.B. is supported by the NRF grant No. RS-2024-00341870. K.K. is supported by the NRF grants Nos. RS-2024-00336346 and RS-2024-00406821. W.L. is supported by the NRF grant No. RS-2024-00453801.

\end{document}